\documentclass[12pt,reqno,oneside]{amsart}
\usepackage{geometry}
\usepackage{xr}

\usepackage[T1]{fontenc}
\usepackage{lmodern}
\usepackage{xcolor}
\usepackage[colorlinks=true, pdfstartview=FitV, linkcolor=blue, citecolor=red,
urlcolor=blue]{hyperref}
\usepackage{enumerate}
\usepackage[numbers]{natbib}
\usepackage{amsmath,amsthm}
\usepackage{graphicx}
\usepackage{longtable}
\usepackage{amssymb}%??????
\usepackage{xypic}
\usepackage{mathrsfs}%????

\newtheorem{thm}{Theorem}[section]
\newtheorem{cor}[thm]{Corollary}
\newtheorem{lem}[thm]{Lemma}
\newtheorem{prop}[thm]{Proposition}

\theoremstyle{definition}

\newtheorem{exmp}[thm]{Example}
\theoremstyle{definition}
\newtheorem{rmk}[thm]{Remark}
\newcommand*\diff{\mathop{}\!\mathrm{d}}

\makeatletter

\newcommand{\Rmnum}[1]{\expandafter\@slowromancap\romannumeral #1@}
\makeatother

\DeclareMathOperator{\ord}{ord}

\usepackage{tikz-cd}
\usepackage{tabu}
\begin{document}
\title[E-O strata on the moduli of  curves]{Ekedahl-Oort strata on the moduli space\\ of  curves of  genus four}
\author{Zijian Zhou}
\maketitle
\thispagestyle{empty}
\begin{abstract}
We study the induced Ekedahl-Oort stratification on the moduli of curves of genus $4$ in positive characteristic.
\end{abstract}
\section{Introduction}
Let $k$ be an algebraically closed field with ${\rm char}(k)=p>0$. 
Let $\mathcal{A}_g\otimes k$ be the moduli space (stack) of principally polarized
abelian varieties of dimension $g$ defined over $k$ and let $\mathcal{M}_g \otimes k$ be the
moduli space of (smooth projective) curves of genus~$g$ defined over $k$.
Ekedahl and Oort introduced a stratification on $\mathcal{A}_g \otimes
k$ consisting of $2^g$ strata, cf. \cite{Oort1999,Ekedahl2009}. These strata are indexed by $n$-tuples
$\mu=[\mu_1, \ldots,\mu_n]$ with $0\leq n \leq g$ and $\mu_1 > \mu_2
>\cdots >\mu_n>0$. The largest stratum is the locus
of ordinary abelian varieties corresponding to the empty $n$-tuple $\mu=\emptyset$.  Their cycle classes have been studied by \cite{vanderGeer1999}.

Via the Torelli map $\tau:
\mathcal{M}_g\otimes k \to \mathcal{A}_g\otimes k$ we can pull back this stratification
to $\mathcal{M}_g \otimes k$ and it is natural to ask what
stratification this provides. Similarly, we can ask for the
induced stratification on the hyperelliptic locus
$\mathcal{H}_g \otimes k$. 
We denote the induced strata on $\mathcal{M}_g\otimes k$ by $Z_{\mu}$. We say a (smooth) curve has Ekedahl-Oort type $\mu$ if the corresponding point in $\mathcal{M}_g\otimes k$ lies in $Z_{\mu}$. 

Here we are interested in the existence of curves of  genus $4$ with given Ekedahl-Oort type. By a curve we mean a smooth irreducible projective curve defined over~$k$.  For $g\leq 3$, we know the situation for the induced Ekedahl-Oort stratification on $\mathcal{M}_g\otimes k$. But for $g\geq 4$ much less is known. Elkin and Pries  \cite{MR3095219} give a 
complete classification for hyperelliptic curves when $p = 2$. Our first result
describes this stratification on $\mathcal{H}_4 \otimes k$ with $p=3$.
In the following we write simply $\mathcal{A}_g$ (resp.\
$\mathcal{M}_g$, $\mathcal{H}_g$)  for $\mathcal{A}_g \otimes k$
(resp.\ $\mathcal{M}_g \otimes k$, $\mathcal{H}_g\otimes k$).
Recall that the indices $\mu$ of the Ekedahl-Oort strata are partially ordered by  $$\mu=[\mu_1,\dots,\mu_n] \preceq \upsilon=[\upsilon_1,\dots,\upsilon_m]$$
if $n\leq m$ and $\mu_i \leq\upsilon_i $ for $i=1,\dots,n$.

\begin{thm}\label{hyper theorem}
Let $k$ be an algebraically closed field with ${{\rm char}(k)=3}$. A smooth hyperelliptic curve of genus $4$ over $k$ has $a$-number $\leq 2$. In particular, $Z_{\mu} \cap \mathcal{H}_4$ is empty if
$\mu \succeq [3,2,1]$. If $\mu \nsucceq  [3,2,1]$, the codimension of $Z_{\mu} \cap \mathcal{H}_4$
in $\mathcal{H}_4$ equals the expected codimension  $\sum_{i=1}^n \mu_i$. Moreover,
in the cases $\mu=[4,1],  [3,1], [3,2],[2,1]$ and $[1]$ the intersection
$Z_{\mu} \cap \mathcal{H}_4$ is irreducible.
\end{thm}
Part of Theorem \ref{hyper theorem} was known.  Frei \cite{SF} proved that hyperelliptic curves in odd characteristic cannot have $a$-number $g-1$. Glass and Pries (\cite[Theorem 1]{Glass2005}) showed that the
intersection of $\mathcal{H}_g$ with the locus $V_l$ of $p$-rank $\leq l$
has codimension $g-l$ in characteristic $p>0$. Pries (\cite[Theorem 4.2]{MR2569747}) showed that
$Z_{[2]} \cap \mathcal{H}_4$ has dimension $5$ for $p\geq 3$.

The following result shows that  certain Ekedahl-Oort strata in $\mathcal{M}_4$ are not empty.
\begin{thm}\label{example E-O [4,2]}
Let $k$ be an algebraically closed field of characteristic $p$. For any odd prime $p$ with $p\equiv \pm 2\,(\, {\rm mod} \, 5)$,
the loci $Z_{[4,2]}$ and $Z_{[4,3]}$ in $\mathcal{M}_4$ are non-empty. For any prime $p\equiv - 1\,(\, {\rm mod} \, 5)$,  there exist superspecial curves of genus $4$ in characteristic $p$. 
\end{thm}
To prove Theorem \ref{example E-O [4,2]}, we use  cyclic covers of the projective line in positive characteristic. 
Furthermore, we  give an alternative but much shorter proof of a result of Kudo \cite{2018arXiv180409063K} showing that there exists a superspecial curve of genus $4$ in characteristic $p$ for all $p$ with $p\equiv 2\,({\, \rm  mod}\,  3)$.
Related results on  Newton polygons of cyclic covers of the projective line and on the existence of curves with given Newton polygon can be found in \cite{2018arXiv180506914L,2018arXiv180504598L}.

\section{Proof of Theorem \ref{hyper theorem}}\label{section proof of theorem 1}
%The $a$-number and a basis of $H^1_{dR}(X)$ }\label{Chapter-EO-section a number and a basis
Let  $X$ be a hyperelliptic curve of genus $4$ defined over $k$ with $p=3$. Before giving the proof of Theorem \ref{hyper theorem}, we prove several propositions needed for Theorem \ref{hyper theorem} and give a basis of the de Rham cohomology of a hyperelliptic curve of genus $4$ defined over $k$.

We first show that any smooth hyperelliptic curve of genus $4$ has $a$-number at most 2. 
\begin{prop}\label{lemma H4 T2 P=3 }
A hyperelliptic curve of genus $4$ in characteristic $3$ has $a$-number at most $2$.
\end{prop}
\begin{proof}
Any smooth hyperelliptic curve $X$ can be written as $y^2=f(x)$ with $f(x)=\sum_{i=0}^9 a_i x^i \in k[x]$ and ${\rm disc}(f)\neq 0$. By putting a branch point at $0$ and by scaling we may assume that $a_1=a_9=1$ and
\begin{equation}\label{hyper equation}
f(x)=x^9+a_8x^8+\dots+a_2x^2+x
\end{equation}
with $a_i\in k$ for $i=2,\dots,8$.
As a basis of $H^0(X,\Omega_X^1)$ we choose
$
 \omega_i={x^{i}}/{y}\diff x$ for  $i=0,\dots,3$. 
Then the Cartier-Manin matrix $H$, i.e. the matrix of the Cartier operator acting on the holomorphic differentials with respect to a basis, of curve $X$~is 
\begin{align}
H=\left ( \begin{array}{lcll}
a_2 & 1 & 0 &0 \\ 
a_5 & a_4 & a_3 &a_2\\
a_8 & a_7 & a_6 &a_5\\
0 & 0 & 1 &a_8\\
\end{array}\right )^{1/3}\, ,\label{Cartier Manin}
\end{align}
where $H^{1/3}=(h_{ij}^{1/3})$ if $H=(h_{ij})$. Since ${\rm rank}(H)\geq 2$, we have for the $a$-number $a=4-{\rm rank}(H)\leq 2 $.
\end{proof}
\begin{rmk}\label{rmk finite fibre}
Note that the map from the parameter space of the $a_i$ ($i=2,\dots,8$) to the hyperelliptic locus has finite fibres.
Indeed, if $\phi$ is an isomorphism between two smooth hyperelliptic curves given by $f_1(x)=\sum_{i=1}^9a_ix^i$ and $f_2(x)=\sum_{i=1}^9b_ix^i$ as in $(\ref{hyper equation})$ that induces an isomorphism of $\mathbb{P}^1$ fixing $0$ and $\infty$, then $\phi$ is given by scaling $x\mapsto \alpha x$ and $y\mapsto \beta y$. We obtain $\alpha ^9/\beta ^2=\alpha /\beta^2=1$ and hence $\alpha^8=1, \beta^2=\alpha$. 
\end{rmk}
We let $Y$ be the open subset of affine space with coordinates $(a_2,\dots,a_8)$ such that ${\rm disc}(f)\neq 0$.
Denote by $T_{a}$  the locus of curves of genus $g$ with $a$-number $\geq a$ in $\mathcal{M}_4$ and by $X_f$ the smooth projective hyperelliptic curve defined by the equation $y^2=f(x)$ as in $(\ref{hyper equation})$. Let $H_f$ be the Cartier-Manin matrix of the curve $X_f$. In the following we simply write $X$ (resp. $H$) for $X_f$ (resp. $H_f$). Now we give a result about the intersection $\mathcal{H}_4\cap T_a$ with $a\leq 2$.

\begin{prop}\label{hyper prop H4 cap T2}
The locus of $\mathcal{H}_4\cap T_a$ with $a\leq 2$ is irreducible of  codimension $a(a+1)/2$. 
\end{prop}
\begin{proof}
For $a=0$, we consider the curve with equation $y^2=f(x)=x^9+tx^5+x$ defined over $k$ where $t$ is a primitive element in $\mathbb{F}_9$. Then ${\rm disc}(f)=2\neq 0$ and by $(\ref{Cartier Manin})$ we have ${\rm rank}(H)=4$.  Hence there is a curve with $a=0$ and note that $\mathcal{H}_4$ is irreducible of dimension $7$. Then by semicontinuity the generic hyperelliptic curve is ordinary and $T_0\cap \mathcal{H}_4$
is irreducible of dimension~$7$.

The condition  $a=1$ means ${\rm rank}(H)=3$. We show that the locus in $Y$ with ${\rm rank}(H)=3$ is given by
$$
(a_8a_6-a_5)(a_2a_4-a_5)+(a_2-a_3a_8)(a_2a_7-a_8)=0\, .
$$

Indeed if $a_2=a_8=0$ and ${\rm disc}(f)\neq 0$, then by Gauss reduction  the rank of $H$ is equal to the rank of 
\begin{align*}
\left ( \begin{array}{lcll}
0 & 1 & 0 &0 \\ 
a_5 & 0 & 0 &0\\
0 & 0 & 0 &a_5\\
0 & 0 & 1 &0\\
\end{array}\right )^{1/3} \, .
\end{align*}
Since we want ${\rm rank}(H)>2$ we must have $a_5\neq 0$. Then this implies  ${\rm rank}(H)=4$ and the curve is ordinary.

Suppose one of $a_2,a_8$ is not zero; by symmetry we can assume $a_2\neq 0$ and the rank of $H$ is equal to the rank of 
\begin{align*}
\left ( \begin{array}{lcll}
a_2 & 1 & 0 &0 \\ 
0 & a_4-a_5/a_2 & a_3 &a_2\\
0 & a_7-a_8/a_2 & a_6 &a_5\\
0 & 0 & 1 &a_8\\
\end{array}\right ) ^{1/3}\, .\\
\end{align*}

We have $\det(H)=0$ as ${\rm rank}(H)=3$ and hence
\begin{align}
(a_8a_6-a_5)(a_2a_4-a_5)+(a_2-a_3a_8)(a_2a_7-a_8)=0\, . \label{T_1}
\end{align}
Note that equation $(\ref{T_1})$ can be rewritten as 
$$
a_7a_2^2 + 2(a_3a_7a_8 + a_4a_5 + 2a_2a_4a_6a_8 + a_2a_8)a_2 + a_3a_8^2 + a_5^2 +
    2a_5a_6a_8=0\, .
$$
This is a $6$-dimensional subspace of $Y$, which is irreducible. Also if we take $a_2=a_7=a_8=1$ and $a_i=0$ for $i\neq 2,7,8$, then ${{\rm disc}(f)=2\neq 0}$ and ${\rm rank}(H)=3$. Hence there is a curve with $a=1$ and by semicontinuity $T_1\cap \mathcal{H}_4$ is irreducible of codimension 1.

For $a=2$, we want to show that the locus in $Y$ with $a=2$ is given by $a_2=a_5=a_8=0$. Since we want ${\rm rank}(H)=2$ and the first and fourth row of $H$ are linearly independent, we have several situations to deal with:
$i)$ $a_2=0$, $ii)$ $a_8=0$ and $iii)$ $a_2a_8\neq 0$.

For the first two cases, if the ${\rm rank}(H)=2$, then the second and third rows of $H$ are linear combinations of the first and fourth rows. Therefore we have $a_2=a_5=a_8=0$. 
For the third case, if $a_2a_8\neq 0$, let $e_i$ to be the $i$-th row of $H$, then with some $b,c,s,t \in k$ we have 
$$
be_1+ce_4=e_2  ,~
se_1+te_4=e_3\, .
$$
This implies $a_3=a_2/a_8, a_4=a_5/a_2, a_6=a_5/a_8, a_7=a_8/a_2$ and hence
$$
f(x)=(x^2+(a_5/a_8)^{1/3}x+(a_2/a_8)^{1/3})^3(x^3+a_8x^2+(a_8/a_2)x)\, ,
$$
which does not have distinct roots, a contradiction. Then we have $a_2=a_5=a_8=0$, which defines an irreducible sublocus in $Y$. Indeed, if we take $a_3=1,a_7=2$ and $a_i=0$ for $i\neq 3,7$, then ${\rm disc}(f)=1\neq 0$ and ${\rm rank}(H)=2$. So we find a curve with ${\rm rank}(H)=2$. Hence by semicontinuity $T_2\cap \mathcal{H}_4$ is irreducible of codimension 3.
\end{proof}

We have seen that any hyperelliptic curve over $k$ with $a$-number $2$ is given by an equation $y^2=f(x)$ as in $(\ref{hyper equation})$ with $(a_2,\dots,a_8) \in Y$ and $a_2=a_5=a_8=0$. We will now use the de Rham cohomology $H^1_{dR}(X)$ for a curve $X$ of genus $g$. Recall that this is a vector space of dimension $2g$ provided with a non-degenerate pairing,  cf. \cite[Section 12]{Oort1999}. Now we consider the action of Verschiebung $V$ on the de Rham cohomology of a curve $X$ given by equation $(\ref{hyper equation})$ with $a$-number $2$.
%\subsection{Basis of $H_{dR}^1(X)$}
First we give a basis of the de Rham cohomology  of a hyperelliptic curve with $a$-number 2. 
Let $X$ be a smooth irreducible complete curve over $k$ with equation 
\begin{align}\label{hyper equation a=2}
y^2=f=x^9+a_7x^7+a_6x^6+a_4x^4+a_3x^3+x,~ a_i\in k\, ,
\end{align}
where the discriminant of $f$ is non-zero. 
Let $\pi : X\to \mathbb{P}^1$ be the hyperelliptic map. Take an open affine cover $\mathcal{U}=\{U_1,U_2\}$ with $U_1=\pi ^{-1}(\mathbb{P}^1-\{0\})$ and $U_2=\pi ^{-1}(\mathbb{P}^1-\{\infty\})$. By Section \cite[Section 5]{MR0241435}, the de Rham cohomology $H^1_{dR}(X)$  can be described as 
$$
H^1_{dR}(X)=Z_{dR}^1(\mathcal{U})/B_{dR}^1(\mathcal{U})
$$
with
$
Z_{dR}^1(\mathcal{U})=\{(t,\omega_1,\omega_2)|~t\in \mathcal{O}_X(U_1\cap U_2),\omega_i\in \Omega_X^1(U_i),\diff t=\omega_1-\omega_2\}
$
and 
$B_{dR}^1(\mathcal{U})=\{(t_1-t_2,\diff t_1,\diff t_2)|~t_i\in \mathcal{O}_X(U_i)\}$. Then $V(H^1_{dR}(X))=H^0(X,\Omega_X^1)$ and $V$ coincides with the Cartier operator on $H^0(X,\Omega_X^1)$. 

For $1\leq i \leq 4$, put $s_i(x)=xf'(x)-2if(x)$ with $f'(x)$ the formal derivative of $f(x)$ and write $s_i(x)=s_i^{\leq i}(x)+s_i^{> i}(x)$ with $s_i^{\leq i}(x)$ the sum of monomials of degree $\leq i$.
By K{\" o}ck and Tait \cite{2017arXiv170903422K}, $H^1_{dR}(X)$ has a basis with respect to $\{U_1,U_2\}$ consisting of the following residue classes with representatives in $Z_{dR}^1(X)$:
\begin{align}
\gamma_i&=[(\frac{y}{x^i},\frac{\psi_i(x)}{2x^{i+1}y}{\rm d}x,-\frac{\phi_i(x)}{2x^{i+1}y}{\rm d}x)], ~ i=1,\dots,4, \label{hyper de rham basis H^1} \\
\lambda_j&=[(0,\frac{x^j}{y}{\rm d}x,\frac{x^j}{y}{\rm d}x)], ~ j=0,\dots,3, \label{hyper de rham basis H^0}
\end{align}
where $\psi_i(x)=s_i^{\leq i}(x)$ and $\phi_i(x)=s_i^{> i}(x)$. Also we have $\langle\gamma_i,\lambda_j \rangle \neq 0$ if $j=i-1$ and $\langle\gamma_i,\lambda_j \rangle = 0$ otherwise. 
Now we give the action of Verschiebung.
\begin{lem}\label{lemma V(basis)}
Let $X$ be a smooth hyperelliptic curve over $k$ with equation $(\ref{hyper equation a=2})$. Let $\{U_1,U_2\}$ be a covering of $X$ as above. Then for the basis of $H^1_{dR}(X)$ given by $(\ref{hyper de rham basis H^1})$ and $(\ref{hyper de rham basis H^0})$ , we have $V(\lambda _0)=V(\lambda _3)=V(\gamma_2)=V(\gamma_3)=0$ and 
\begin{align*}
V(\lambda _1)&=a_7^{1/3}\lambda_2+a_4^{1/3}\lambda_1+\lambda_0\, ,
V(\lambda _2)=\lambda_3+a_6^{1/3}\lambda_2+a_3^{1/3}\lambda_1\, ,\\
V(\gamma_1)&=\lambda_2+a_6^{1/3}\lambda_1+a_3^{1/3}\lambda_0,
V(\gamma_4)=a_4^{1/3}\lambda_2+(1-(a_3a_7)^{1/3}+(a_4a_6)^{1/3})\lambda_1+a_6^{1/3}\lambda_0\, .
\end{align*}
\end{lem}
\begin{proof}
Since $V$ coincides with the Cartier operator on $H^0(X,\Omega_X^1)$, we have $V(h{\rm d}x)=(-{\rm d}^{2}h/{\rm d}x^{2})^{1/3} {\rm d}x$ with $h\in k(x)$. We will compute $V(\gamma_1)$ and the rest of the lemma will follow easily by using a similar argument. Note that we always have for $1\leq i \leq 4$ 
\begin{align*}
V(\frac{\psi_i(x)}{2x^{i+1}y}{\rm d}x)=V(-\frac{\phi_i(x)}{2x^{i+1}y}{\rm d}x)
\end{align*}
as $0=V({\rm d}(y/x^i))=V(\frac{\psi_i(x)}{2x^{i+1}y}{\rm d}x)-V(-\frac{\phi_i(x)}{2x^{i+1}y}{\rm d}x)$. So it suffices to compute $V(\frac{\psi_i(x)}{2x^{i+1}y}{\rm d}x)$ instead of computing $V(\gamma_i)$. For $i=1$, we have 
\begin{align*}
V(\frac{\psi_1(x)}{2x^{1+1}y}{\rm d}x)=V(\frac{1}{xy}{\rm d}x)=(-\frac{{\rm d}({\rm d}(\frac{1}{xy})/{\rm d}x)}{dx})^{1/3}\diff x\, .
\end{align*}
\iffalse
with
\begin{align*}
\diff (\frac{1}{xy})/{\rm d}x=\frac{-x\frac{{\rm d}y}{{\rm d}x}-y}{x^2y^2}
=-\frac{1}{x^2y}+\frac{f'(x)}{xy^3}\, .
\end{align*}
\fi
Note that ${\rm d}f'(x)/{\rm d}x=0$ and by a calculation we have
\iffalse 
\begin{align*}
{\rm d}(-\frac{1}{x^2y})/{\rm d}x&=\frac{2xy+x^2\frac{{\rm d}y}{{\rm d}x}}{x^4y^2}=-(\frac{1}{x^3y}+\frac{f'(x)}{x^2y^3})\, ,\\
{\rm d}(\frac{f'(x)}{xy^3})/{\rm d}x&=-\frac{y^3f'(x)}{x^2y^6}=-\frac{f'(x)}{x^2y^3}\, .
\end{align*}
Then we obtain
$$
=-(-\frac{1}{x^3y}-\frac{f'(x)}{x^2y^3}-\frac{f'(x)}{x^2y^3})=\frac{y^2-xf'(x)}{x^3y^3}
$$
\fi
\begin{align*}
-\frac{{\rm d}({\rm d}(\frac{1}{xy})/{\rm d}x)}{\diff x}=\frac{x^9+a_6x^6+a_3x^3}{x^3y^3}\, .
\end{align*}
Hence 
$
V(\psi_1(x)/({2x^{1+1}y}){\rm d}x)=(x^2/{y}+a_6^{1/3}x/y+a_3^{1/3}/y)\diff x\, 
$
and we have $V(\gamma_1)=\lambda_2+a_6^{1/3}\lambda_1+a_3^{1/3}\lambda_0$.
\end{proof}

Now we give a proof of Theorem \ref{hyper theorem}.
\begin{proof}[Proof of Theorem \ref{hyper theorem}]
The theorem holds for cases $\mu=[0]$ and $[1]$ by Proposition \ref{hyper prop H4 cap T2} where we showed that $T_a\cap \mathcal{H}_4$ is irreducible with codimension $a(a+1)/2$ for $a\leq 2$. Also  $T_3\cap \mathcal{H}_4$ is empty by Proposition \ref{lemma H4 T2 P=3 }, hence $Z_{\mu}\cap \mathcal{H}_4=\varnothing$ for any $\mu \succeq[3,2,1]$.

We only need to prove  that the theorem is true for the remaining nine Ekedahl-Oort strata, that is six strata consisting of curves with $a$-number 2 and three strata consisting of curves with $a$-number 1.

As an outline of the proof, we first show that for $\mu=[2,1],[3,1],[3,2],$ $[4,1],[4,2]$ and $[4,3]$ the codimension of $Z_{\mu} \cap \mathcal{H}_4$
in $\mathcal{H}_4$ equals the expected codimension  $\sum_{i=1}^n \mu_i$ with $\mu=[\mu_1,\dots,\mu_n]$. For $\mu=[2],[3]$ and $[4]$, combined with the fact $V_l\cap \mathcal{H}_4$ is non-empty of codimension $4-l$ in $\mathcal{H}_4$ for $l=0,1,2$ by Glass and Pries \cite[Theorem 1]{Glass2005}, the intersection $Z_{\mu}\cap \mathcal{H}_4$ also has the expected codimension. In the cases $\mu=[2,1],[3,1],[3,2],[4,1],[4,2]$ and $[4,3]$, a curve with Ekedahl-Oort type $\mu$ can be written as equation $(\ref{hyper equation a=2})$ by the proof of Proposition \ref{hyper prop H4 cap T2}. 

Throughout the proof, denote by $X$ a smooth hyperelliptic curve given by equation $y^2=f(x)$ as in $(\ref{hyper equation a=2})$ with Ekedahl-Oort type $\mu$. 
Denote 
$$
Y_2:= V(H^0(X,\Omega_X^1))=V(\langle\lambda _0,\dots,\lambda _3 \rangle )\text{ and } Y_6:=Y_2^{\perp}
$$
  with respect to the paring $\langle \, , \, \rangle $ on $H^1_{dR}(X)$. Put $v:\{0,1,\dots,8\}\to \{0,1,\dots,4\}$  the final type of $X$. From Lemma \ref{lemma V(basis)} above, we know that $Y_2$ is a $2$-dimensional subspace of $H^0(X,\Omega_X^1)$ generated by $V(\lambda_1)$ and $V(\lambda_2)$.

Let $\mu=[2,1]$. By Proposition $\ref{hyper prop H4 cap T2}$ the intersection  $T_2\cap \mathcal{H}_4$ is irreducible of dimension 4. For any curve $X$ corresponding to a point in $\mathcal{H}_4\cap T_2$,  we have by Lemma~\ref{lemma V(basis)}
\begin{align}
V(Y_2)&=\langle V^2(\lambda_1),V^2(\lambda_2) \rangle  \nonumber=\langle V(a_7^{1/3}\lambda_2+a_4^{1/3}\lambda_1+\lambda_0),V(\lambda_3+a_6^{1/3}\lambda_2+a_3^{1/3}\lambda_1) \rangle \nonumber\\
&=\langle a_7^{1/9}V(\lambda_2)+a_4^{1/9}V(\lambda_1),a_6^{1/9}V(\lambda_2)+a_3^{1/9}V(\lambda_1) \rangle\, . \label{hyper lemma Y_2 basis}
\end{align}
We consider the curve associated to $(a_3,a_4,a_6,a_7)=(1,0,0,2)$. Then ${{\rm disc}(f)=1\neq 0}$. Moreover, $V^n(Y_2)=Y_2$ for any $n\in \mathbb{Z}_{>0}$. Hence the semi-simple rank of $V$ acting on $H^0(X,\Omega_X^1)$ is $2$ and the Ekedahl-Oort type of this curve is $[2,1]$. Since the $p$-rank can only decrease under specialization, the generic point of $\mathcal{H}_4\cap T_2$ has Ekedahl-Oort type $[2,1]$ and $Z_{[2,1]}\cap \mathcal{H}_4$ is irreducible of dimension 4.

Now we move to the case $\mu=[3,1]$. We show that a curve with Ekedahl-Oort type $[3,1]$ has equation $(\ref{hyper equation a=2})$ with $a_7a_3=a_6a_4$ and ${\rm disc}(f)=a_3a_4^2+a_6a_7+1\neq 0$. Then the irreducibility and dimension will follow naturally.

Suppose $X$ has Ekedahl-Oort type $[3,1]$, then $X$ is given by equation $(\ref{hyper equation a=2})$ with $\dim(V(Y_2))=1$. Then by Lemma \ref{lemma V(basis)} and relation $(\ref{hyper lemma Y_2 basis})$, we have $a_3a_7=a_4a_6$. 

Put $Y_3:=V(Y_6)$ then we have 
$$
\dim Y_3=v(6)=v(2)+4-2=3
$$ 
by the properties of the final type $v$. If we take $(a_3,a_4,a_6,a_7)=(0,1,0,0)$, then ${\rm disc}(f)=1\neq 0$. Note that $V(\gamma_1)=\lambda _2$ and $V(\gamma _4)=\lambda_2+\lambda_1$, hence $Y_2=\langle \lambda _3,\lambda _1+\lambda _0 \rangle $ by the Lemma \ref{lemma V(basis)}. Furthermore, it is easy to check that $V^2(\lambda_2)=0$ and $V^n(\lambda_1)=\lambda_1$. Then we get $v(1)=1$. Also there exists an element $\gamma=\sum_{i=1}^3 b_i \gamma _i$ with $b_i\in k$ in $Y_6$ such that $b_1\neq 0$, otherwise it contradicts that $\langle \gamma,\lambda _0+\lambda _1 \rangle =0$. Thus $b_1^{1/3}\lambda_2= V(\gamma)\in Y_3$ and by simple computation we have $\dim V(Y_3)=2$. Then there is a curve with Ekedahl-Oort type $[3,1]$ and by semicontinuity we have the $Z_{[3,1]}\cap \mathcal{H}_4$ is irreducible of dimension 3.

Let $\mu=[3,2]$, we show that the smooth hyperelliptic curve $X$ with Ekedahl-Oort type $[3,2]$ can be written as 
\begin{align}\label{hyper equation [3,2]}
y^2=f(x)=x^9+a_7x^7+\alpha ^3a_7x^6+a_4x^4+\alpha ^3a_4x^3+x
\end{align}
with $a_4,a_7,\alpha \in k^*$ satisfying $\alpha^3 a_7^2+\alpha a_7=a_4+\alpha a_4^2$ and ${{\rm disc}(f)=(a_4\alpha + a_7\alpha^{2} + 1)^9\neq0}$.

Indeed, if the curve $X$ is given by equation $(\ref{hyper equation a=2})$ with Ekedahl-Oort type $[3,2]$, then we have $v(2)=1$ and $v(1)=1$. By Lemma \ref{lemma V(basis)} and relation $(\ref{hyper lemma Y_2 basis})$, the condition $v(2)=1$ implies $a_3a_7=a_4a_6$. Also $Y_6$ is generated by $\lambda_i$ for $i=0,\dots,3$ and $\sum_{j=1}^4b_j\gamma_j$ with $b_j \in k$ and $\langle\sum_{j=1}^4b_j\gamma_j,Y_2\rangle=0$. 

If $a_7=0$, by $a_3a_7=a_4a_6$ we have $(i)~a_6=0$ or $(ii)~a_4=0$. 

If we suppose $a_7=a_6=0$, then $Y_2=\langle a_4^{1/3}\lambda_1+\lambda
_0,\lambda_3+a_3^{1/3}\lambda_1 \rangle $ and $Y_6$ is generated by $\lambda_i$ and $\sum_{j=1}^4b_j\gamma_j$ with 
\begin{equation}\label{equation b1b2b3b4}
b_1+b_2a_4^{1/3}=b_2a_3^{1/3}+b_4=0,~b_1,\dots,b_4\in k\, .
\end{equation}
Write $Y_3=V(Y_6)$. If $a_4=0$, then we have $V^2(Y_2)=\langle \, 0\,\rangle$, a contradiction since $X$ has Ekedahl-Oort type $[3,2]$. Now suppose $a_4\neq 0$, then for all $b_1,b_4$ satisfying $(\ref{equation b1b2b3b4})$ we have
 $$
V(Y_6)=Y_3=\langle Y_2,(1+a_4^{2/9}a_3^{1/9})\lambda_2+(\frac{a_3}{a_4})^{1/9}\lambda_1+a_3^{1/3}\lambda_0)\rangle 
\, . $$
 Since $v(3)=1$, we have 
 $
1+a_4^{2/9}a_3^{1/9}=0
 $, 
which implies $a_3^1a_4^2=-1$. In this case we have ${\rm disc}(f)=a_3^3a_4^6+1=0$, a contradiction.

Now if $a_7=a_4=0$, $Y_6$ is generated by $\lambda_i$ and $\sum_{j=1}^4b_j\gamma_j$ with $b_1=b_2a_3^{1/3}+a_6^{1/3}b_3+b_4=0$.
By Lemma \ref{lemma V(basis)} we have $V(\gamma_4)=\lambda_1+a_6^{1/3}\lambda_0$, hence
\begin{align*}
Y_3=V(Y_6)&=\langle Y_2,V(b_2\gamma_2+\dots+b_4\gamma_4) \rangle =\langle Y_2,V(\gamma_4) \rangle \\
&=\langle \lambda_0,\lambda_3+a_6^{1/3}\lambda_2+a_3^{1/3}\lambda_1,\lambda_1+a_6^{1/3}\lambda_0 \rangle \, .
\end{align*}
Therefore we have $V(Y_3)=Y_2$, a contradiction with~${v(3)=1}$.

Now assume $a_7\neq 0$ and put $\alpha=(a_6/a_7)^{1/3}$. Then we have $a_3=\alpha^3a_4$ by relation $a_7a_3=a_6a_4$, and 
\begin{align}
Y_2=\langle  a_7^{1/3}\lambda_2+a_4^{1/3}\lambda_1+\lambda_0,\lambda_3+\alpha a_7^{1/3}\lambda_2+\alpha a_4^{1/3}\lambda_1\rangle\, .\label{Y_2 [3,2]}
\end{align}
By a similar argument to the above, $Y_6$ is generated by $\lambda_i$ and $\sum_{j=1}^4b_j\gamma_j$ with $\langle \sum_{j=1}^4b_j\gamma_j, Y_2 \rangle=0$, this implies
 $$b_4-\alpha b_1=b_3a_7^{1/3}+b_2a_4^{1/3}+b_1=0, b_i\in k\, .$$ Then $Y_3=V(Y_6)=\langle Y_2,V(b_1\gamma_1+b_4\gamma_4) \rangle$ and this equals $\langle Y_2,\xi \rangle$ with 
\begin{align*}
\xi=(1+\alpha^{1/3}a_4^{1/3})\lambda_2+(\alpha a_7^{1/3}+\alpha ^{1/3})\lambda_1+(\alpha^{1/3}a_4^{1/3}+\alpha^{4/3}a_7^{1/3})\lambda_0 \rangle \, .
\end{align*}
Since $X$ has Ekedahl-Oort type $\mu=[3,2]$, we have $v(3)=1$. Then $V(\langle\, \xi\, \rangle)=V(Y_2)=V(a_7^{1/3}\lambda_2+a_4^{1/3}\lambda_1)$ by relation $(\ref{Y_2 [3,2]})$ and Lemma \ref{lemma V(basis)}. Thus we have 
\begin{align*}
\alpha a_7^{2/3}+(\alpha a_7)^{1/3}=a_4^{1/3}+\alpha^{1/3}a_4^{2/3}
\end{align*}
and hence
\begin{align}\label{prop [3,2] special}
\alpha^3 a_7^2+\alpha a_7=a_4+\alpha a_4^2\, .
\end{align}
If $\alpha=0$, by equality $(\ref{prop [3,2] special})$ we have $a_4=0$ and in equation $(\ref{hyper equation [3,2]})$ we have $f=x^9+a_7x^7+x$ and one can easily show that $v(1)=0$, a contradiction as $X$ has Ekedahl-Oort type $\mu=[3,2]$. By a similar argument we can prove $a_4\neq 0$.
If we take $(a_7,\alpha,a_4)=(2,2,1)$ in equation $(\ref{hyper equation [3,2]})$, we have ${\rm disc}(f)=2\neq 0$. 
 Then there is a curve with Ekedahl-Oort type $[3,2]$ and by semicontinuity we have $Z_{[3,2]}\cap \mathcal{H}_4$ is irreducible of dimension 2.
 
Let $\mu=[4,1]$. We show that any smooth hyperelliptic curve with Ekedahl-Oort type $[4,1]$ can be 
written as
\begin{align}
y^2=f(x)=x^9+a_7x^7+\alpha ^{3} a_7x^6-\alpha^9 a_7x^4-\alpha ^{12}a_7x^3+x \label{hyper equation [4,1]}
\end{align}
with $a_7,\alpha \in k^*$ and ${\rm disc}(f)=2\alpha^{10}a_7 + \alpha^2a_7 + 1\neq 0$. Then it will follow that $Z_{[4,1]}\cap \mathcal{H}_4$ is irreducible of dimension $2$. 
Indeed, if $X$ is given by equation $(\ref{hyper equation a=2})$ with Ekedahl-Oort type [4,1], then $v(2)=1,v(1)=0$ and by Lemma \ref{lemma V(basis)} and relation $(\ref{hyper lemma Y_2 basis})$ we have $a_3a_7=a_4a_6$ . 

$a):$ If $a_7=0$, we have $a_6=0$ or $a_4=0$. Assume $a_6=a_7=0$, then by relation $(\ref{hyper lemma Y_2 basis})$ we have $V(Y_2)=\langle V(\lambda_1)\rangle$. By Lemma \ref{lemma V(basis)}, we have $a_4=0$ since the $p$-rank of $X$ is zero. But then $X$ has Ekedahl-Oort type $[4,2]$ by a similar argument with $Y_6=Y_2^{\perp}$ and $Y_3=V(Y_6)$ as  in case $\mu=[3,2]$. Now suppose $a_7=a_4=0$. We have $a_6=0$ since $X$ has $p$-rank 0. Then again $X$ has  Ekedahl-Oort type $[4,2]$. 

$b):$ Now assume $a_7\ne 0$. Put $\alpha=(a_6/a_7)^{1/3}$ and we have $a_3=\alpha ^3a_4$ by equation $a_7a_3=a_6a_4$. Write $Y_1=V(Y_2)=\langle a_7^{1/9}V(\lambda_2)+a_4^{1/9}V(\lambda_1) \rangle $. Suppose we have $V^{m}(Y_1)=0$ and $V^{m-1}(Y_1)\neq 0$ for some $m\in \mathbb{Z}_{>0}$. For $0\leq i\leq m$,  put $V^i(Y_1)=\langle g_i(\lambda_0,\lambda_3)+c_i\lambda_2+d_i\lambda_1 \rangle $ with $g_i(x,y)\in k[x,y]$. Then we have 
\begin{align}
c_i=(\alpha c_{i-1}^{1/3}+d_{i-1}^{1/3})a_7^{1/3}, d_i=(\alpha c_{i-1}^{1/3}+d_{i-1}^{1/3})a_4^{1/3} \label{induction assumption c_i d_i}
\end{align}
for $1\leq i\leq m$. Now $V(V^{m-1}(Y_1))=0$. Therefore by Lemma \ref{lemma V(basis)} we have 
$$
V^{m}(Y_1)=V(\langle c_{m-1}\lambda_2+d_{m-1}\lambda_1+g_{m-1}(\lambda_0,\lambda_3) \rangle )=0\, .
$$
Hence we have $c_m=d_m=0$ as $V(\langle \lambda_0,\lambda_3 \rangle )=0$ by Lemma \ref{lemma V(basis)}. Thus we obtain $(\alpha c_{m-1}^{1/3}+d_{m-1}^{1/3})a_7^{1/3}=(\alpha c_{m-1}^{1/3}+d_{m-1}^{1/3})a_4^{1/3}=0$, which implies $\alpha c_{m-1}^{1/3}+d_{m-1}^{1/3}=0$ as $a_7\neq 0$. Using the inductive assumption $(\ref{induction assumption c_i d_i})$ again, we have 
$$
\alpha c_{m-1}^{1/3}+d_{m-1}^{1/3}=((\alpha ^{3}a_7^{1/3}+a_4^{1/3})(\alpha^{1/3}c_{m-2}+d_{m-2}))^{1/3}=0\, .
$$
Since $V^{m-1}(Y_1)\neq 0$, we have $\alpha^{1/3}c_{m-2}+d_{m-2}\neq 0$ and hence $(\alpha ^{3}a_7^{1/3}+a_4^{1/3})=0$. This implies $a_4=-\alpha^9 a_7$ and $a_3=\alpha^3a_4=-\alpha^ {12} a_7$. Now we  compute $Y_3=V(Y_6)$ and this equals
$$
\langle Y_2,(1-\alpha^{10/3}a_7^{1/3})\lambda_2+(\alpha a_7^{1/3}+\alpha^{1/3})\lambda_1+g(\lambda_0,\lambda_3)  \rangle 
$$
for some $g(x,y) \in k[x,y]$. Combined with 
\begin{align*}
Y_2&=\langle a_7^{1/3}(\lambda_2-\alpha^3\lambda_1)+\lambda_0,\lambda_3+\alpha a_7^{1/3}(\lambda_2-\alpha^3\lambda_1) \rangle \\
&=\langle \lambda_3-\alpha \lambda_0, a_7^{1/3}(\lambda_2-\alpha^3\lambda_1)+\lambda_0 \rangle\, ,
\end{align*}
we have $v(3)=1$ if 
$$
\alpha^3(-1+\alpha^{10/3}a_7^{1/3})=(\alpha a_7^{1/3}+\alpha^{1/3}) \, ,
$$ this is equivalent to $\alpha^3(\alpha^{16}-1)a_7=\alpha^9+\alpha$. Otherwise $X$ has Ekedahl-Oort type $[4,1]$ for general pair $(a_7,\alpha)\in \mathbb{A}_k^2$. Hence we have the desired conclusion for $\mu=[4,1]$. 
Moreover if in equation $(\ref{hyper equation [4,1]})$ we take $(a_7,\alpha)=(v^{10},v^9)$ with $v$ a primitive element in $\mathbb{F}_{27}$, then ${\rm disc}(f)=v^{21}\neq0$ and there is a curve with equation $(\ref{hyper equation [4,1]})$ has Ekedahl-Oort type associated to $\mu=[4,1]$. Hence by semicontinuity we have proved the theorem is true for $\mu=[4,1]$.

For $\mu=[4,2]$, from the discussion in the case $[4,1]$ above, a hyperelliptic curve 
  $X$ with Ekedahl-Oort type $[4,2]$ is either given by equation $(\ref{hyper equation a=2})$ with 
$ a_7=a_6=a_4=0$, or it can be written as 
$$
y^2=f(x)=x^9+a_7x^7+\alpha ^{3} a_7x^6-\alpha^9 a_7x^4-\alpha ^{12}a_7x^3+x
$$
with $a_7,\alpha \in k, a_7\neq 0$ satisfying $\alpha^3(\alpha^{16}-1)a_7=\alpha^9+\alpha$ and  ${\rm disc}(f)\neq 0$. Moreover,the curve with equation $y^2=x^9+x^7+x$ has ${\rm disc}(f)=1\neq 0$ and Ekedahl-Oort type $\mu=[4,2]$. Hence $Z_{[4,2]}\cap \mathcal{H}_4$ is non-empty of dimension $1$.

For $\mu=[4,3]$, a curve $X$ with Ekedahl-Oort type $[4,3]$ is given by $(\ref{hyper equation a=2})$ with $V(Y_2)=\langle 0\rangle$. Then by Lemma \ref{lemma V(basis)} we have $a_7=a_6=a_4=a_3=0$. This implies $X$ is isomorphic to the curve with equation
$
y^2=x^9+x
$. 
Now we have proved the theorem for $\mu=[2,1],[3,1],[3,2],[4,1],[4,2]$ and $[4,3]$. Also for $\mu=[2],[3]$ and $[4]$, by Glass and Pries \cite[Theorem 1]{Glass2005} the intersection $V_l\cap \mathcal{H}_4$ has codimension $4-l$ in $\mathcal{H}_4$ for $l=0,1,2$.  Since we have showed that $Z_{[2,1]}$ (resp. $Z_{[3,1]}$ and $Z_{[4,1]}$) intersects $ \mathcal{H}_4$ with codimension $3$ (resp. 4 and 5), it follows that $Z_{\mu}\cap \mathcal{H}_4$ has the expected codimension for $\mu=[2],[3]$ and $[4]$. 
\end{proof}
\iffalse
\begin{exmp}
$1)$ A hyperelliptic curve $y^2=x^9+t^{10}x^7+t^{11}x^6+x^4+tx^3+x$ defined over $k$ with $t$ a primitive element of $F_{27}$ has Ekedahl-Oort type $[4,1]$.\\
$2)$  A hyperelliptic curve $y^2=x^9+a_7x^7+x$ defined over $k$ with $a_7\in k^*$ has Ekedahl-Oort type $[4,2]$.
\end{exmp}
\fi
\section{Proof of Theorem $\ref{example E-O [4,2]}$}
We prove Theorem $\ref{example E-O [4,2]}$ using cyclic covers of the projective line in characteristic $p>0$. First we introduce some general facts on cyclic covers of the projective line and give a basis of the first de Rham cohomology for them. 
% Furthermore, we  give an alternative proof of a result of Kudo \cite{2018arXiv180409063K} showing that there exists a superspecial curve of genus $4$ in characteristic $p$ for all $p$ with $p\equiv 2\,({\, \rm  mod}\,  3)$.

Let $k$ be an algebraically closed field of characteristic $p>0$. 
We fix  an integer $m\geq 2$ with $p\nmid m$. Write $a=(a_1,\dots, a_N)$ for an $N$-tuple of positive integers with $N\geq 3$. 
We say  $a$ is a monodromy vector of length $N$ if 
\begin{align}
\sum_{i=1}^Na_i\equiv 0~({\rm mod}~m ),\qquad \gcd(a_i,m)=1,\, i=1,\dots, N. \label{relation monodromy vector}
\end{align}
  There is an action of $ (\mathbb{Z}/m\mathbb{Z})^*\times \mathfrak{S}_N$ on the set of monodromy vectors of length $N$, where the symmetric group $\mathfrak{S}_N$ acts by permutation of indices and $(\mathbb{Z}/m\mathbb{Z})^*$ acts by multiplication on vectors. Two monodromy vectors $a$ and $a'$ are called equivalent if they are in the same orbit.

Let $P_1,\dots,P_N$ be the distinct points in $\mathbb{P}^1$ and $x$ be a coordinate on $\mathbb{P}^1$. By a change of coordinates, we may assume $P_1=0$ and $P_N=\infty$. Denote by $x-\xi_i$ with $\xi_i\in k$ the local parameter of $P_i$ ($\xi_1=0$) for $1\leq i \leq N-1$. We consider a smooth projective curve $X$ given by equation
\begin{align}
y^m=f_a(x)=\Pi_{i=1}^{N-1}(x-\xi_i)^{a_i}\, .\label{cyclic equation}
\end{align}
Note that the isomorphism class of the curve depends only  on the orbit of monodromy vector $a$.
For $N=3$, the supersingularity of cyclic covers of the projective line has been studied and examples of supersingular curves was given for $4\leq g\leq 11$, see \cite{2018arXiv180504598L}.
In \cite{ELKIN20111}, Elkin gave a bound for the $a$-number of $X$ for $m\geq 2$ and $N\geq 3$.	 

A curve defined by equation $(\ref{cyclic equation})$ is equipped with a $\mathbb{Z}/m\mathbb{Z}$ action generated  by $\epsilon:\, (x,y)\mapsto (x,\zeta^{-1} y)$ with $\zeta\in k$ a primitive $m$-th root of unity. This $\epsilon$ also induces an automorphism on $H^0(X,\Omega_X^1)$. Then we can decompose
 \begin{align}
 H^0(X,\Omega_X^1)=\bigoplus_{n=1}^{m-1}H^0(X,\Omega_X^1)_{(n)}\, ,\label{decomp H^0}
 \end{align}
 where $H^0(X,\Omega_X^1)_{(n)}:=\{\omega\in H^0(X,\Omega_X^1)~|~\epsilon^*(\omega)=\zeta^n \omega\}$ is the $\zeta^n$-eigenspace of $H^0(X,\Omega_X^1)$. Denote by $\langle z \rangle:=z-\lfloor z\rfloor$ the fractional part of $z$ for any $z\in \mathbb{R}$. Put 
\begin{align*}
 b(i,n)&:=\lfloor (na_i)/m \rfloor,\, 
 \omega_{n,l}:=y^{-n}x^l\Pi_{i=1}^{N-1}(x-\xi_i)^{b(i,n)}\diff x\, . 
\end{align*} 
Then (see for example \cite{MR2735989}) the space $H^0(X,\Omega_X^1)_{(n)}$ is generated by $\omega_{n,l}$  with $0\leq l \leq -2+\sum_{i=1}^N\langle na_i/m\rangle$ and 
 $$
 d_n=\dim H^0(X,\Omega_X^1)_{(n)}=-1+\sum_{i=1}^{N}\langle na_i/m\rangle\, .
 $$
By the Hurwitz formula, the curve has genus 
$
g=1+{((N-2)m-N)}/{2}\, .
$

Now we introduce a basis of de Rham cohomology of a curve $X$ given by equation $(\ref{cyclic equation})$. 
Fix a monodromy vector $a$ together with an ordered $N$-tuple $(\xi_1,\dots,\xi_N)$.
 Let $\pi : X\to \mathbb{P}^1$ be the $\mathbb{Z}/m\mathbb{Z}$-cover. Put $U_1=\pi ^{-1}(\mathbb{P}^1-\{0\})$ and $U_2=\pi ^{-1}(\mathbb{P}^1-\{\infty\})$. For the open affine cover $\mathcal{U}=\{U_1,U_2\}$, we consider the de Rham cohomology $H^1_{dR}(X)$ similar to the hyperelliptic case in Section \ref{section proof of theorem 1} above, i.e.
$$
H^1_{dR}(X)\cong Z_{dR}^1(\mathcal{U})/B_{dR}^1(\mathcal{U})
$$
with
$
Z_{dR}^1(\mathcal{U})=\{(t,\omega_1,\omega_2)|t\in \mathcal{O}_X(U_1\cap U_2),\omega_i\in \Omega_X^1(U_i),\diff t=\omega_1-\omega_2\}
$
and 
$B_{dR}^1(\mathcal{U})=\{(t_1-t_2,\diff t_1,\diff t_2)|t_i\in \mathcal{O}_X(U_i)\}$.

By $(\ref{decomp H^0})$, the vector space $H^0(X,\Omega_X^1)$ is generated by $\omega_{n,l}$ with $1\leq n\leq m-1$ and $0\leq l \leq -2+\sum_{i=1}^N\langle na_i/m\rangle$. Moreover for the basis of $H^1(X,\mathcal{O}_X)$, we have the following result.
\begin{lem}\label{lem basis H^1(X,OX)}
Let $K(X)$ be the function field of $X$. For all integers $n,l$ such that $1\leq n\leq m-1$ and $0\leq l \leq -2+\sum_{i=1}^N\langle na_i/m\rangle$, the elements $f_{n,l}:=y^nx^{-l-1}\Pi_{i=1}^{N-1}(x-\xi_i)^{-b(i,n)}\in K(X)$ are regular on $U_1\cap U_2$ and their residue classes $[f_{n,l}]$ form a basis of $H^1(X,\mathcal{O}_X)$ with respect to $\{U_1,U_2\}$.
\end{lem}
\begin{proof}
It suffices to show that the $f_{n,l}$ are regular on $U_1\cap U_2$ for all integers $n,l$ such that $\omega_{n,l}\in H^0(X,\Omega_X^1)$. One can check the linear independence by checking the order of $f_{n,l}$ at $\infty$. 

Note that for $P_i$ with $i=2,\dots,N-1$, we have 
\begin{align*}
\ord_{P_i}(\frac{y^n}{x^{l+1}\Pi_{i=1}^{N-1}(x-\xi_i)^{b(i,n)}})&=n\ord_{P_i}(y)-b(i,n)\ord_{P_i}(x-\xi_i)\\
&=na_i-mb(i,n)=na_i-m\lfloor \frac{na_i}{m}\rfloor \geq 0\, .
\end{align*}
Then $f_{n,l}$ is regular on $U_1\cap U_2$ for $1\leq n\leq m-1$ and $0\leq l \leq -2+\sum_{i=1}^N\langle na_i/m\rangle$.
\end{proof}
Put
$
s_a(x):=\Pi_{i=1}^{N-1}(x-\xi_i)^{b(i,n)}
$.
Denote by $h_a(x)$ the polynomial in $k[x]$ such that 
$$
\frac{nxs_a(x)f_a'(x)+((l+1)s_a(x)+xs_a'(x))f_a(x)}{s_a^2(x)}=\Pi_{i=1}^{N-1}(x-\xi_i)^{a_i-b(i,n)-1}h_a(x)\, ,
$$
where $f'_a(x)$ (resp. $s'_a(x)$) is the formal derivative of $f_a(x)$ (resp. $s_a(x)$).
In the following we write simply $s(x)$ (resp. $f(x)$, $h(x)$) for $s_a(x)$ (resp. $f_a(x)$, $h_a(x)$).

Then we have the following result.
\begin{prop}\label{de Rham basis cyclic cover of P1}
Let $X$ be a  smooth projective curve over $k$ given by equation $(\ref{cyclic equation})$.
Then $H^1_{dR}(X)$ has a basis with respect to $\mathcal{U}=\{U_1,U_2\}$ consisting of the following residue classes with representatives in $Z_{dR}^1(\mathcal{U})$:
\begin{align}
\alpha_{n,l}&=[(0,\omega_{n,l},\omega_{n,l})], \,1\leq n\leq m-1,\, 0\leq l \leq -2+\sum_{i=1}^N\langle na_i/m\rangle\, ,\\
\beta_{n,l}&=[(f_{n,l},\frac{\psi_{n,l} (x)t(x)}{x^{l+2}y^{m-n}}{\rm d}x,-\frac{\phi_{n,l} (x)t(x)}{x^{l+2}y^{m-n}}{\rm d}x)]\, ,\label{de rham basis AS H^1}
\end{align}
where $t(x)=\Pi_{i=1}^{N-1}(x-\xi_i)^{a_i-b(i,n)-1}$ and $\psi(x)+\phi(x)=h(x)$ with $\psi(x)$  the sum of monomials of degree $\leq l+1$.
\end{prop}
\begin{proof}
We use the exact sequence 
\begin{align*}
0\to H^0(X,\Omega_X^1)\to H^1_{dR}(X)\to H^1(X,\mathcal{O}_X)\to 0\, .
\end{align*}
The elements $\alpha_{n,l}$ are images of $\omega_{n,l}$ under the canonical map.

By Lemma \ref{lem basis H^1(X,OX)}, we have $[f_{n,l}]\in H^1(X,\mathcal{O}_X)$ for any $n,l$ such that $\omega_{n,l}\in H^0(X,\Omega_X^1)$. To prove the theorem, we need to show that the elements $\beta_{n,l}$ are well defined and are mapped to the element $[f_{n,l}]$ in $H^1(X,\mathcal{O}_X)$. We first show that  ${\psi_{n,l} (x)}/{(x^{l+2}y^{m-n})}\in \mathcal{O}(U_1)$ and ${\phi_{n,l} (x)}/{(x^{l+2}y^{m-n})}\in \mathcal{O}(U_2)$. Next we show 
$$
\diff f_{n,l}=\frac{\psi_{n,l} (x)}{x^{l+2}y^{m-n}}-(-\frac{\phi_{n,l} (x)}{x^{l+2}y^{m-n}})
$$
and then we will have the desired conclusion.

Note that for any $P_i$ with $i=1,\dots,N-1$, we have 
%&=(a_i-b(i,n)-1)\ord_{P_i} (x-\xi_i)-(m-n)\ord_{P_i} (y)+m-1\\
\begin{align*}
\ord_{P_i}(\frac{t(x)}{y^{m-n}}\diff x)&=\ord_{P_i}(\frac{\Pi_{i=1}^{N-1}(x-\xi_i)^{a_i-b(i,n)-1}}{y^{m-n}})+\ord_{P_i} (\diff x)\\
&=m(a_i-\lfloor \frac{na_i}{m}\rfloor-1)-a_i(m-n)+m-1=m\langle \frac{na_i}{m}\rangle-1\geq 0\, ,
\end{align*}
since $\gcd (a_i,m)=1$. Hence $$\ord_{P_i}(\frac{\phi_{n,l} (x)t(x)}{x^{l+2}y^{m-n}}{\rm d}x))\geq \ord_{P_i}(\frac{t(x)}{y^{m-n}}\diff x)\geq 0$$ and ${\phi_{n,l} (x)t(x)}/{(x^{l+2}y^{m-n})}{\rm d}x$ is regular at $P_i$ for $i=2,\dots,N-1$. By a similar argument, ${\psi_{n,l} (x)t(x)}/{(x^{l+2}y^{m-n})}{\rm d}x$ is also regular at $P_i$ for $i=2,\dots,N-1$. For $P_1=0$, we have 
\begin{align*}
\ord_{P_1}(\frac{\phi_{n,l}(x)t(x)}{x^{l+2}y^{m-n}}\diff x)\geq \ord_{P_1}(\frac{t(x)}{y^{m-n}}\diff x)\geq 0\, ,
\end{align*}
since all the monomials of $\phi_{n,l}(x)$ has degree $\geq l+2$. Then the residue class of the element ${\phi_{n,l} (x)t(x)}/{(x^{l+2}y^{m-n})}{\rm d}x$ is regular on $U_2$.  For $P_N=\infty$, by a similar calculation we have
\begin{align*}
\ord_{P_N}(\frac{\psi_{n,l}(x)t(x)}{x^{l+2}y^{m-n}}\diff x)\geq \ord_{P_N}(\frac{x^{l+1}t(x)}{x^{l+2}y^{m-n}}\diff x)\geq 0\, .
\end{align*}
\iffalse
The right hand side is equal to
%&=-1+((m-n)\sum_{i=1}^{N-1}a_i-m(N-1)-m\sum_{i=1}^{N-1}\langle \frac{na_i}{m}\rangle)+(m-n)\sum_{i=1}^{N-1}a_i\\
\begin{align*}
&\ord_{P_N}(1/x\diff x)+\ord_{P_N}(\frac{t(x)}{y^{m-n}})=m(N-1)-1-m\sum_{i=1}^{N-1}\langle \frac{na_i}{m}\rangle\geq 0\, .
\end{align*}
\fi
Let $n,l$ be  integers  such that $\omega_{n,l}\in H^0(X,\Omega_X^1)$. Then
\begin{align*}
\diff (f_{n,l})=\diff (\frac{y^n}{x^{l+1}s(x)})&=\frac{nxs(x)f'(x)+((l+1)s(x)+xs'(x))f(x)}{x^{l+2}s^2(x)y^{m-n}}\diff x\\
&=\frac{t(x)h(x)}{x^{l+2}y^{m-n}}\diff x=\frac{t(x)\psi_{n,l} (x)}{x^{l+2}y^{m-n}}\diff x-\frac{-t(x)\phi_{n,l}(x)}{x^{l+2}y^{m-n}}\diff x\, .
\end{align*}
\end{proof}

\begin{rmk}
The pairing $\langle~, ~\rangle$ for this basis is as follows: $\langle\alpha_{i_1,j_1},\beta_{i_2,j_2} \rangle \neq 0$ if $(i_1,j_1)=(i_2,j_2)$ and $\langle\alpha_{i_1,j_1},\beta_{i_2,j_2} \rangle = 0$ otherwise. Indeed, for $(i_1,j_1)=(i_2,j_2)$ we have $\ord_{\infty}(1/x\diff x)=-1$ and hence $\langle\alpha_{i_1,j_1},\beta_{i_2,j_2} \rangle \neq  0$. For the other cases,  the proof is similar to the proof of \cite[Theorem 4.2.1]{soton373877}.
\end{rmk}

Now for $p=2$ and $N=4$, we have the following
\begin{cor}\label{cor a=(1,1,1,m-3)}
Let $k$ be an algebraically closed field of characteristic $p$ and $a$ be a monodromy vector satisfying relation $(\ref{relation monodromy vector})$ with $a=(1,1,1,m-3)$. Let $X$ be a curve of genus $m-1$ over $k$ given by equation
$$
y^m=x(x-1)(x-\xi)\, ,
$$
where $\xi\neq 0,1\in k$.
Then $H^1_{dR}(X)$ has a basis with respect to $\mathcal{U}=\{U_1,U_2\}$ consisting of the following residue classes with representatives in $Z_{dR}^1(\mathcal{U})$:
\begin{align*}
\alpha_{i,0}&=[(0,\frac{1}{y^i}\diff x,\frac{1}{y^i}\diff x)], \,\frac{m}{3}< i\leq m-1\, ,
\alpha_{j,1}=[(0,\frac{x}{y^j}\diff x,\frac{x}{y^j}\diff x)], \,\frac{2m}{3}< j\leq m-1\, ,\\
\beta_{i,0}&=[(\frac{y^i}{x},\frac{\xi}{xy^{m-i}}{\rm d}x,-\frac{x+(\xi+1)}{y^{m-i}}{\rm d}x)], \,\text{ $i$ even and }\frac{m}{3}< i\leq m-1,\\
\beta_{i,0}&=[(\frac{y^i}{x},0,-\frac{(\xi+1)}{y^{m-i}}{\rm d}x)], \,\text{ $i$ odd and }\frac{m}{3}< i\leq m-1,\\
\beta_{j,1}&=[(\frac{y^j}{x^2},0,0)], \,\text{ $j$ even and }\frac{2m}{3}< j\leq m-1,\\
\beta_{j,1}&=[(\frac{y^j}{x^2},\frac{\xi}{x^2y^{m-j}}{\rm d}x,-\frac{1}{y^{m-j}}{\rm d}x)], \,\text{ $j$ odd and }\frac{2m}{3}< j\leq m-1,
\end{align*}

\end{cor}
\begin{proof}
Note that $a=(a_1,a_2,a_3,a_4)=(1,1,1,m-3)$. Then by definition $b(i,n)=\langle na_i/m\rangle=0$ for any $1\leq n\leq  m-1$ and $1\leq i\leq 3$. Moreover,  the differential form $\omega_{n,l}=y^{-n}x^l\diff x$ is holomorphic for $1\leq n \leq m-1$ if and only if $0\leq l\leq -2+\sum_{i=1}^4\langle na_i/m \rangle\leq -2+3=1$. If $0\leq n \leq m/3$, then $\sum_{i=1}^4\langle na_i/m \rangle=1$ and $H^0(X,\Omega_X^1)_{(n)}=\langle 0 \rangle$. The rest of the corollary follows from Proposition \ref{de Rham basis cyclic cover of P1}.
\end{proof}
We can now give the proof of Theorem \ref{example E-O [4,2]}. 

\begin{proof}[Proof of Theorem $\ref{example E-O [4,2]}$]
Let $X$ be a curve of genus $4$ with equation 
\begin{align*}
y^5=x(x-1)(x-\xi)\, ,
\end{align*}
where $\xi\in k\backslash\mathbb{F}_{p}$ and $p= 2\,(\,{\rm mod}\, 5)$.
We first show that  $X$  has Ekedahl-Oort type $[4,2]$. Then  the case $p= -2\, (\,{\rm mod}\, 5)$ is similar and hence we omit it. Write $p=5r+2$ with $r\in \mathbb{Z}_{\geq 0}$. Since $p$ is a prime, either $r=0$ or $r$ is an odd positive integer.

Let $Y_8:=H^1_{dR}(X)$ and $Y_4:=H^0(X,\Omega_X^1)$. By Corollary \ref{cor a=(1,1,1,m-3)}, we have $V(Y_8)=Y_4$ and $Y_4=\langle \alpha_{2,0},\alpha_{3,0},\alpha_{4,0},\alpha_{4,1}  \rangle$.

Note that 
\begin{align*}
\mathcal{C}(\frac{1}{y^2}\diff x)&=\mathcal{C}(\frac{y^{5r}}{y^{5r+2}}\diff x)=\frac{1}{y}\mathcal{C}((x(x-1)(x-\xi))^{r}\diff x)=0\, .
\end{align*}
Then similarly
$\mathcal{C}({1}/{y^3}\diff x)=\mathcal{C}((x(x-1)(x-\xi))^{4r+1}\diff x)/{y^4}\, ,
\mathcal{C}({1}/{y^4}\diff x)= \mathcal{C}((x(x-1)(x-\xi))^{2r}\diff x)/{y^2}$ and $
\mathcal{C}({x}/{y^4}\diff x)=\mathcal{C}(x^{2r+1}((x-1)(x-\xi))^{2r}\diff x)/{y^2}$.
One can show that the coefficient of $x^{p-1}=x^{5r+1}$ in $x^{2r+i}((x-1)(x-\xi))^{2r}$ cannot be simultaneously zero  for $i=0,1$ and $\xi \in k\backslash \mathbb{F}_{p}$. Similarly, the coefficients of $x^{p-1}$ and $x^{2p-1}$ in $(x(x-1)(x-\xi))^{4r+1}$ are both not zero.
Then  $Y_2:=V(Y_4)=\langle \alpha_{2,0} ,\, \gamma\alpha_{4,0}+\eta \alpha_{4,1} \rangle$ with $\gamma,\eta\in k^*$. Denote by $Y_6=Y_2^{\perp}$ the orthogonal complement with respect to the pairing  on $H^1_{dR}(X)$. Hence by a calculation using Corollary \ref{cor a=(1,1,1,m-3)}, we have 
$$
Y_6=\langle \alpha_{2,0},\alpha_{3,0},\alpha_{4,0},\alpha_{4,1},\beta_{3,0},\lambda_0\beta_{4,0}+\lambda_1\beta_{4,1}\rangle\, ,
$$
where $\lambda_i\in k^*$. This implies $Y_3:=V(Y_6)=\langle \alpha_{2,0},\alpha_{4,0},\alpha_{4,1}\rangle$ and $V(Y_3)=\langle  \alpha_{2,0}\rangle$. We obtain that $X$ has Ekedahl-Oort type $[4,2]$.

Now we show that $Z_{[4,3]}$ is non-empty in $\mathcal{M}_4$ for any odd prime $p$ with $p\equiv \pm 2(\, {\rm mod} \, 5)$.

Take now  $m=5$ and monodromy vector $a=(1,1,1,2)$ in equation $(\ref{cyclic equation})$ and consider a curve $X$ given by equation 
\begin{align}
y^5=x(x-\xi)(x+\xi)\, ,\label{curve [4,3]}
\end{align}
where $\xi\in k^*$. For $p\neq 2,5$, the curve is of genus $4$.  Moreover by Lemma \ref{lem basis H^1(X,OX)}, the vector space $H^0(X,\Omega_X^1)$ has a basis given by forms $y^2\diff x,y^3\diff x,y^4\diff x$ and $xy^4\diff x$. Now if $p> 2$ and $p\equiv  2(\, {\rm mod} \, 5)$, then write $p=5r+2$ with $r$ an odd positive integer and we consider the action of the Cartier operator $\mathcal{C}$. By a similar calculation as in the case $Z_{[4,2]}$, we have
\begin{align*}
\mathcal{C}(\frac{\diff x}{y^2})=\mathcal{C}(\frac{x\diff x}{y^4})=0,\,
\mathcal{C}(\frac{\diff x}{y^3})=\eta_1\frac{x\diff x}{y^4}, \,
\mathcal{C}(\frac{\diff x}{y^4})=\eta_2\frac{\diff x}{y^2}
\end{align*}
with some $\eta_1,\eta_2\in k^*$.
Then  $X$ has  Ekedahl-Oort type $[4,3]$ and $v(2)=0$ with $v$ the final type, cf. \cite{Oort1999,vanderGeer1999}. This implies $X$ is supersingular, see \cite[page 1379]{10.2307/23030376}. By a similar argument, one can show that for $p\equiv  3(\, {\rm mod} \, 5)$, the curve has Ekedahl-Oort type $[4,3]$ and hence is supersingular.

Now we show the existence of superspecial curves of genus $4$ in characteristic $p\equiv - 1(\, {\rm mod} \, 5)$. Again let $X$ be the same curve given by equation $(\ref{curve [4,3]})$ and write $p=5r+4$ with $r$ some positive integer. Then
\begin{align*}
\mathcal{C}(\frac{\diff x}{y^2})=\mathcal{C}(\frac{y^{15r+10}}{y^{15r+10+2}}\diff x)=\frac{1}{y^3}\mathcal{C}(x^{3r+2}(x^2-v^2)^{3r+2}\diff x)=0.
\end{align*}
 By a similar calculation, we have
$
\mathcal{C}({\diff x}/{y^3})=
\mathcal{C}({x\diff x}/{y^4})=
\mathcal{C}({\diff x}/{y^4})=0
$.
Then $\mathcal{C}(H^0(X,\Omega_X^1))=0$ and $X$ is superspecial.
\end{proof}

\section{An alternative proof of Kudo's result}

In \cite{2018arXiv180409063K}, Kudo showed that there is a superspecial non-hyperelliptic curve of genus $4$ over $k$ for any  odd prime $p\equiv 2\,({\, \rm  mod}\,  3)$ by viewing such curves as an intersection of a quadric and a cubic in $\mathbb{P}^3$.
Using our approach,  we can give an alternative proof of Kudo's result. 
\begin{prop}\cite[Theorem $3.1$]{2018arXiv180409063K}
There exists a superspecial curve of genus $4$ in characteristic $p\equiv 2\,({\, \rm  mod}\,  3)$.
\end{prop}
\begin{proof}

Consider the monodromy vector $a=(1,1,1,1,1,1)$. Let $X$ be the smooth projective curve of genus $4$ with equation 
$$
y^3=x(x-\xi)(x-\xi^3)(x-\xi^5)(x-\xi^7)=x^5+x\, ,
$$
where $\xi\in k$ is a primitive $8$-th root of unity.

By Lemma \ref{lem basis H^1(X,OX)}, the vector space $H^0(X,\Omega_X^1)$ has a basis consisting of forms $1/y\diff x$, $1/y^2\diff x,x/y^2\diff x$ and $x^2/y^2\diff x$. Write $p=3r+2$ with $r$ an odd positive integer. Then 
\begin{align*}
\mathcal{C}(\frac{1}{y}\diff x)=\mathcal{C}(\frac{y^{6r+3}}{y^{6r+3+1}}\diff x)=\frac{1}{y^2}\mathcal{C}(x^{2r+1}(x^4+1)^{2r+1}\diff x)\, .
\end{align*}
Since $r$ is an odd integer,  the coefficient of $x^{np-1}$ in $x^{2r+1}(x^4+1)^{2r+1}$ for any  $n\in \mathbb{Z}_{>0}$ is zero. Similarly, we have $\mathcal{C}(x^i\diff x/y^2)=0$ for $i=0,1,2$.
\iffalse
\begin{align*}
\mathcal{C}(\frac{1}{y^2}\diff x)=\mathcal{C}(\frac{y^{3r}}{y^{3r+2}}\diff x)=\frac{1}{y}\mathcal{C}(x^{r}(x^4+1)^{r}\diff x)=0\, ,\\
\mathcal{C}(\frac{x}{y^2}\diff x)=\mathcal{C}(\frac{xy^{3r}}{y^{3r+2}}\diff x)=\frac{1}{y}\mathcal{C}(x^{r+1}(x^4+1)^{r}\diff x)=0\, ,\\
\mathcal{C}(\frac{x^2}{y^2}\diff x)=\mathcal{C}(\frac{x^2y^{3r}}{y^{3r+2}}\diff x)=\frac{1}{y}\mathcal{C}(x^{r+2}(x^4+1)^{r}\diff x)=0\, .
\end{align*}
Since the coefficients of $x^{p-1}$ in $x^{r+i}(x^4+1)^{r}$ are zero for $i=0,1,2$.
\fi
\end{proof}

%\bibliography{/Users/zhouzijian/Dropbox/bibfile}
\bibliographystyle{plain}

\end{document}